\documentclass[10pt]{article}

\usepackage{amsmath,amssymb,amsthm}
\usepackage{mathrsfs}
\usepackage{cite}

\newtheorem{theorem}{Theorem}[section]
\newtheorem{corollary}[theorem]{Corollary}
\newtheorem{proposition}[theorem]{Proposition}
\theoremstyle{definition}
\newtheorem{definition}[theorem]{Definition}
\newtheorem{remark}[theorem]{Remark}

\newcommand{\N}{\mathbb{N}}
\newcommand{\R}{\mathbb{R}}
\newcommand{\LL}{\mathcal{L}}
\newcommand{\CC}{\mathscr{C}}
\newcommand{\W}{\textrm W}
\renewcommand{\L}{\textrm{L}}
\newcommand{\E}{\textrm E}

\newcommand{\Ir}{{_tI_b^\alpha}}
\newcommand{\IlC}{{_aI_t^{1-\alpha}}}
\newcommand{\IrC}{{_tI_b^{1-\alpha}}}
\newcommand{\Dl}{{_a^CD_t^\alpha}}
\newcommand{\Dr}{{_t^CD_b^\alpha}}

\newenvironment{keywords}{\begin{center}
\begin{minipage}[c]{11.2cm}
{\bf Keywords:}} {\end{minipage}
\end{center}}

\newenvironment{msc}{\begin{center}
\begin{minipage}[c]{11.2cm}
{\bf 2010 Mathematics Subject Classification:}}
{\end{minipage}
\end{center}}

\begin{document}

\title{Duality for the left and right\\
fractional derivatives\thanks{This is a preprint of a paper whose final 
and definite form will appear in the international journal \emph{Signal Processing}, 
ISSN 0165-1684. Paper submitted Dec/2013; revised Apr, July and Sept 2014; 
accepted for publication 18/Sept/2014.}}

\author{M. Cristina Caputo$^{1}$\\
\texttt{caputo@utexas.edu}
\and Delfim F. M. Torres$^{2}$\\
\texttt{delfim@ua.pt}}

\date{$^{1}$Department of Mathematics\\
The University of Texas at Austin\\
Austin, TX 78712-1202, USA\\[0.3cm]
$^{2}$CIDMA, Department of Mathematics\\
University of Aveiro\\
3810-193 Aveiro, Portugal}

% ----------------------------------------

\maketitle

% ----------------------------------------

\begin{abstract}
We prove duality between the left and right fractional derivatives,
independently on the type of fractional operator.
Main result asserts that the right derivative of a function is the dual
of the left derivative of the dual function or, equivalently,
the left derivative of a function is the dual
of the right derivative of the dual function.
Such duality between left and right fractional operators
is useful to obtain results for the left operators
from analogous results on the right operators and vice versa.
We illustrate the usefulness of our duality theory
by proving a fractional integration by parts formula
for the right Caputo derivative and by proving
a Tonelli-type theorem that ensures the existence
of minimizer for fractional variational problems
with right fractional operators.
\end{abstract}

\begin{msc}
26A33, 49J05.
\end{msc}

\begin{keywords}
fractional derivatives and integrals, duality theory,
fractional integration by parts, calculus of variations,
existence of solutions.
\end{keywords}

% ----------------------------------------

\section{Introduction}

Differential equations of fractional order appear in many branches of physics, mechanics
and signal processing. Roughly speaking, fractional calculus deals with derivatives and integrals of noninteger order.
The subject is as old as the calculus itself. In a letter correspondence of 1695,
L'Hopital proposed the following problem to Leibniz: ``Can the meaning of derivatives
with integer order be generalized to noninteger orders?'' Since then,
several mathematicians studied this question,
among them Liouville, Riemann, Weyl and Letnikov.
An important issue is that the fractional derivative of order $\alpha$ at a point $x$
is a local property only when $\alpha$ is an integer. For noninteger cases,
the fractional derivative at $x$ of a function $f$ is a nonlocal operator,
depending on past values of $f$ (left derivatives) or future values of $f$
(right derivatives). In physics, if $t$ denotes the time-variable, the right
fractional derivative of $f(t)$ is interpreted as a future state of the
process $f(t)$. For this reason, the right-derivative is usually neglected in
applications, when the present state of the process does not depend on the
results of the future development. However, right-derivatives are unavoidable
even in physics, as well illustrated by the fractional variational calculus
\cite{book:Baleanu,book:FCV}. Consider a variational principle,
which gives a method for finding signals that minimize or maximize
the value of some quantity that depend upon those signals.
Two different approaches in formulating differential equations
of fractional order are then possible: in the
first approach, the ordinary (integer order) derivative in a differential equation
is simply replaced by the fractional derivative. In the second approach,
one modifies the variational principle by replacing the integer order derivative
by a fractional one. Then, minimization of the action leads
to the differential equation of the system. This second
approach is considered to be, from the standpoint of applications, the more sound one:
see, e.g., \cite{MR2332922,MR2898042}.
It turns out that this last approach introduces necessarily
right derivatives, even when they are not present in the formulation
and data of the problems: see, e.g., \cite{MR2752428,MR2943359}. Indeed,
left and right derivatives are linked by the following
integration by parts formula:
$$
\int_a^b d^+ f \cdot g \, dt =- \int_a^b f\cdot d^- g \, dt  ,
$$
over the set of functions $f$ and $g$ admitting right
and left derivatives, here generally represented by
$d^+$ and $d^-$ respectively,
and such that $f(a) g(a) =f(b)g(b)=0$
(cf. Corollary~\ref{cor4.1} and Remark~\ref{rem4.2}).
Our duality results provide a very elegant way to deal with
such fractional problems, where there exists an interplay between
left and right fractional derivatives of the signals.

There are many fields of applications where one can use the fractional calculus \cite{Machado}.
Examples include viscoelasticity, electrochemistry, diffusion processes,
control theory, heat conduction, electricity, mechanics,
chaos and fractals (see, \textrm{e.g.},
\cite{Blutzer,Jonsson,Lorenzo,Mainardi,Miller,Podlubny,Sabatier,Samko}).
A large (but not exhaustive) bibliography
on the use of fractional calculus in linear viscoelasticity
may be found in the book \cite{MR2676137}. Recently, a lot of attention
has been put on the fractional calculus of variations
(see, \textrm{e.g.}, \cite{AGRA3,Almeida2,Atanackovic,Dreisigmeyer1,%
Dreisigmeyer2,Frederico:Torres07,Klimek1,Klimek2,Muslih1,Muslih2,MyID:207,Riewe1,Riewe2}).
We also mention \cite{Almeida1}, were necessary and sufficient conditions of optimality
for functionals containing fractional integrals and fractional derivatives are presented.
For results on fractional optimal control see, e.g., \cite{AGRA0,Frederico:Torres08,PATFree}.
In the present paper we work mainly with the Caputo fractional derivative.
For problems of calculus of variations with boundary conditions,
Caputo's derivative seems to be more natural because,
for a given function $y$ to have continuous Riemann--Liouville fractional derivative
on a closed interval $[a,b]$, the function must satisfy the conditions $y(a)=y(b)=0$ \cite{Almeida2}.
We also mention that, if $y(a)=0$, then the left Riemann--Liouville derivative of $y$ of order
$\alpha \in (0,1)$ is equal to the left Caputo derivative;
if $y(b)=0$, then the right Riemann--Liouville derivative of $y$ of order
$\alpha \in (0,1)$ is equal to the right Caputo derivative.

The paper is organized as follows. In Section~\ref{sec2} we
present the definitions of fractional calculus needed in the sequel.
Section~\ref{sec3} is dedicated to our original results: we introduce
a duality theory between the left and right fractional operators.
It turns out that the duality between
the left and right fractional derivatives or integrals
is independent of the type of fractional operator:
a right operator applied to a function can always be computed as the dual
of the left operator evaluated on the dual function or, equivalently,
we can compute the left operator applied to a function as the dual
of the right operator evaluated on the dual function. We claim
that such duality is very useful, allowing one to
directly obtain results for the right operators
from analogous results on the left operators, and vice versa.
This fact is illustrated in Section~\ref{sec4}, where we
show the usefulness of our duality theory in the fractional calculus of variations.
Due to fractional integration by parts, differential equations
containing right derivatives are common in the fractional variational theory
even when they are not present in the data of the problems. Here
we use our duality argument to obtain a fractional integration by parts formula
for the right Caputo derivative (Section~\ref{sub:sec:IBP});
and we show conditions assuring the existence of minimizers
for fractional variational problems with right fractional operators
(Section~\ref{sub:sec:VP}). We end with Section~\ref{sec:conc} of conclusions.

Many different dualities are available in the literature. Indeed,
duality is an important general theme that has manifestations
in almost every area of mathematics, with numerous different meanings.
One can say that the only common characteristic of duality, between those different meanings,
is that it translates concepts, theorems or mathematical structures into other concepts,
theorems or structures, in a one-to-one fashion. For example,
\cite{Caputo} introduces the concept of duality between two different
approaches to time scales: the delta approach, based on the forward
$\sigma$ operator, and the nabla approach, based on the backward $\rho$ operator
\cite{MR1843232}. There is, however, no direct connection between
the time-scale calculus considered in \cite{Caputo},
which is a theory for unification of difference equations (of integer order)
with that of differential equations (of integer order), and the
the fractional (noninteger order) calculus now considered.
We are not aware of any published work on the concept
of duality as we do here.

% ----------------------------------------

\section{Brief review on fractional calculus}
\label{sec2}

There are several definitions of fractional derivatives and fractional integrals,
like Riemann--Liouville, Caputo, Riesz, Riesz--Caputo, Weyl, Grunwald--Letnikov,
Hadamard, Chen, etc. We present the definitions of the first two of them.
Except otherwise stated, proofs of results may be found in \cite{Kilbas}
(see also \cite{book:Ortigueira}).

Let $f:[a,b]\rightarrow\mathbb{R}$ be a function, $\alpha$ a positive real number,
$n$ the integer satisfying $n-1 < \alpha \le n$, and $\Gamma$ the Euler gamma function.
Then, the left Riemann--Liouville fractional integral of order $\alpha$ is defined by
$$
{_aI_x^\alpha}f(x)=\frac{1}{\Gamma(\alpha)}\int_a^x (x-t)^{\alpha-1}f(t)dt
$$
and the right Riemann--Liouville fractional integral of order $\alpha$ is defined by
$$
{_xI_b^\alpha}f(x)=\frac{1}{\Gamma(\alpha)}\int_x^b(t-x)^{\alpha-1} f(t)dt.
$$
The left and right Riemann--Liouville fractional derivatives of
order $\alpha$ are defined, respectively, by
$$
{_aD_x^\alpha}f(x)=\frac{d^n}{dx^n} {_aI_x^{n-\alpha}}f(x)
=\frac{1}{\Gamma(n-\alpha)}\frac{d^n}{dx^n}\int_a^x (x-t)^{n-\alpha-1}f(t)dt
$$
and
$$
{_xD_b^\alpha}f(x)=(-1)^n\frac{d^n}{dx^n} {_xI_b^{n-\alpha}}f(x)
=\frac{(-1)^n}{\Gamma(n-\alpha)}\frac{d^n}{dx^n}
\int_x^b (t-x)^{n-\alpha-1} f(t)dt.
$$
The left and right Caputo fractional derivatives of order $\alpha$
are defined, respectively, by
$$
{_a^CD_x^\alpha}f(x) = {_aI_x^{n-\alpha}}\frac{d^n}{dx^n} f(x)
=\frac{1}{\Gamma(n-\alpha)}\int_a^x (x-t)^{n-\alpha-1}f^{(n)}(t)dt
$$
and
$$
{_x^CD_b^\alpha}f(x)=(-1)^n{_xI_b^{n-\alpha}}\frac{d^n}{dx^n} f(x)
=\frac{1}{\Gamma(n-\alpha)}\int_x^b(-1)^n (t-x)^{n-\alpha-1} f^{(n)}(t)dt.
$$

\begin{remark}
\label{rem:2.1}
When $\alpha=n$ is integer, the fractional operators reduce to:
\begin{equation*}
\begin{split}
{_aI_x^{n}}f(x)
&= \int_a^x d\tau_1 \int_a^{\tau_1} d\tau_2 \ldots \int_a^{\tau_{n-1}}f(\tau_n) d\tau_n,\\
{_xI_b^{n}}f(x)
&= \int_x^b d\tau_1 \int_{\tau_1}^b d\tau_2 \ldots \int_{\tau_{n-1}}^b f(\tau_n) d\tau_n,\\
{_aD_x^{n}}f(x) &= {_a^CD_x^{n}}f(x)  \, =  \, f^{(n)}(x),\\
{_xD_b^{n}}f(x) &= {_x^CD_b^{\alpha}}f(x) \, = \, (-1)^n f^{(n)}(x).
\end{split}
\end{equation*}
\end{remark}

There exists a relation between the Riemann--Liouville
and the  Caputo fractional derivatives:
$$
{_a^CD_x^\alpha}f(x)={_aD_x^\alpha}f(x)-\sum_{k=0}^{n-1}
\frac{f^{(k)}(a)}{\Gamma(k-\alpha+1)}(x-a)^{k-\alpha}
$$
and
$$
{_x^CD_b^\alpha}f(x)={_xD_b^\alpha}f(x)
-\sum_{k=0}^{n-1}\frac{f^{(k)}(b)}{\Gamma(k-\alpha+1)}(b-x)^{k-\alpha}.
$$
Therefore,
$$
\mbox{if } f(a)=f'(a)=\cdots=f^{(n-1)}(a)=0,
\mbox{ then } {_a^CD_x^\alpha}f(x)={_aD_x^\alpha}f(x)
$$
and
$$
\mbox{if } f(b)=f'(b)=\cdots=f^{(n-1)}(b)=0,
\mbox{ then } {_x^CD_b^\alpha}f(x)={_xD_b^\alpha}f(x).
$$
These fractional operators are linear, \textrm{i.e.},
$$
\mathcal{P} (\mu f(x)+\nu g(x))
=\mu\, \mathcal{P}f(x)+\nu \,\mathcal{P}g(x),
$$
where $\mathcal{P}$ is ${_aD_x^\alpha}$, ${_xD_b^\alpha}$, ${_a^CD_x^\alpha}$,
${_x^CD_b^\alpha}$, ${_aI_x^\alpha}$ or ${_xI_b^\alpha}$,
and $\mu$ and $\nu$ are real numbers.

If $f \in C^n[a,b]$, then the left and right Caputo derivatives are continuous on $[a,b]$.
The main advantage of Caputo's approach is that the initial conditions for fractional
differential equations with Caputo derivatives take on the same
form as for integer-order differential equations.

% ----------------------------------------

\section{Duality of the left and right derivatives}
\label{sec3}

We show that there exists a duality between
the left and right fractional operators.
We begin by introducing the notion of dual function.

\begin{definition}[Dual function]
\label{def:dualF}
Let $f:[a,b]\rightarrow\mathbb{R}$.
Then its dual function, denoted by $f^{\star}$,
is defined as $f^{\star}:[-b,-a]\rightarrow\mathbb{R}$ by
$$
f^{\star}(x)=f(-x).
$$
\end{definition}

\begin{remark}
\label{rem:dual:dual}
A direct consequence of the definition of dual function is that
$$
f^{\star\star}(x)=f(x).
$$
\end{remark}

The next result asserts that the left fractional integral of a function $f$
is the right fractional integral for its dual function $f^{\star}$.

\begin{theorem}[Duality of the left and right fractional integrals]
\label{fractional_duality}
Let $f:[a,b]\rightarrow\mathbb{R}$ be a function
and $\alpha$ a positive real number. Then,
$$
{_aI_x^\alpha}f(x)={_{-x}I_{-a}^\alpha}f^{\star}(-x).
$$
\end{theorem}

\begin{proof}
By definition, the left Riemann--Liouville
fractional integral of order $\alpha$ is
$$
{_aI_x^\alpha}f(x)=\frac{1}{\Gamma(\alpha)}
\int_a^x (x-t)^{\alpha-1}f(t)dt.
$$
By the simple change of coordinates $s=-t$,
\begin{equation*}
\begin{split}
{_aI_x^\alpha}f(x)
&= -\frac{1}{\Gamma(\alpha)}\int_{-a}^{-x} (x+s)^{\alpha-1}f(-s)ds\\
&=\frac{1}{\Gamma(\alpha)}\int_{-x}^{-a} (s-(-x))^{\alpha-1}f^{\star}(s)ds\\
&={_{-x}I_{-a}^\alpha}f^{\star}(-x).
\end{split}
\end{equation*}
This concludes the proof.
\end{proof}

A consequence of Theorem~\ref{fractional_duality}
is the following duality for the Caputo
and Riemann--Liouville fractional derivatives.

\begin{theorem}[Duality of the left and right Caputo fractional derivatives]
\label{thm:C:FD}
Let $f:[a,b]\rightarrow\mathbb{R}$ be a function
and $\alpha$ a positive real number. Then,
$$
{{_a^CD}_x^\alpha}f(x)={{^C_{-x}D}_{-a}^\alpha}f^{\star}(-x).
$$
\end{theorem}

\begin{proof}
Set $g(x)=\frac{d^n}{dx^n}f(x)$.
We observe that $g^{\star}(-x)=(-1)^n \frac{d^n}{dx^n}f^{\star}(-x)$.
By definition of the left Caputo derivative,
$$
{{_a^CD}_x^\alpha}f(x)= {_aI_x^{n-\alpha}}g(x).
$$
Hence, by Theorem~\ref{fractional_duality}, we have
\begin{equation*}
\begin{split}
{_aI_x^{n-\alpha}}g(x)&={_{-x}I_{-a}^{n-\alpha}}g^{\star}(-x)\\
&=(-1)^n {_{-x}I_{-a}^{n-\alpha}} \frac{d^n}{dx^n}f^{\star}(-x)\\
&={{^{C}_{-x}D}_{-a}^\alpha}f^{\star}(-x).
\end{split}
\end{equation*}
In the last equality we have used the definition
of right Caputo derivative for the dual function
$f^{\star}$ at the point $-x$.
\end{proof}

\begin{theorem}[Duality of the left and right Riemann--Liouville derivatives]
\label{thm:RL:FD}
Let $f:[a,b]\rightarrow\mathbb{R}$ be a function
and $\alpha$ a positive real number. Then,
$$
{_aD_x^\alpha}f(x)={_{-x}D_{-a}^\alpha}f^{\star}(-x).
$$
\end{theorem}

\begin{proof}
The left  Riemann--Liouville fractional derivative of order
$\alpha$ for the function $f$ is, by definition,
$$
{_aD_x^\alpha}f(x)=\frac{d^n}{dx^n}  {_aI_x^{n-\alpha}}f(x).
$$
By definition, the right Riemann--Liouville fractional derivative
of order $\alpha$ for the function $f^{\star}(s)$ at $-x$ is
$$
{_{-x}D_{-a}^\alpha}f^{\star}(-x)
=(-1)^n\frac{d^n}{ds^n}{_{-x}I_{-a}^\alpha}f^{\star}(-x).
$$
By Theorem~\ref{fractional_duality},
$$
{_aI_x^\alpha}f(x)={_{-x}I_{-a}^\alpha}f^{\star}(-x).
$$
Therefore,
$$
{_aD_x^\alpha}f(x)=\frac{d^n}{dx^n}{_{-x}I_{-a}^\alpha}f^{\star}(-x).
$$
Using the change of coordinates $s=-x$,
$$
{_aD_x^\alpha}f(x)=(-1)^n\frac{d^n}{ds^n}{_{-x}I_{-a}^\alpha}f^{\star}(-x).
$$
It follows, by definition, that
$$
{_aD_x^\alpha}f(x)={_{-x}D_{-a}^\alpha}f^{\star}(-x).
$$
This concludes the proof.
\end{proof}

A consequence of Theorems~\ref{thm:C:FD} and \ref{thm:RL:FD} is the following corollary,
which is valid for any of the fractional derivatives considered in this article.

\begin{corollary}
\label{cor:mr}
The right derivative of the dual function of a given function
$f$ is the dual of the left derivative of $f$.
\end{corollary}

\begin{remark}
Having in mind Remark~\ref{rem:dual:dual}, one can write,
in an equivalent way, Corollary~\ref{cor:mr} as:
\begin{quotation}
The right derivative of a given function $f$ is the dual
of the left derivative of the dual of $f$.
\end{quotation}
\end{remark}

% ----------------------------------------

\section{Application of our main result}
\label{sec4}

The duality introduced in Section~\ref{sec3},
between left and right fractional operators,
can be used to obtain results for the left operators,
from analogous results on the right operators, and vice versa.
We illustrate the usefulness of our duality argument
with a fractional integration by parts formula
(Section~\ref{sub:sec:IBP}), which is the fundamental tool
to establish necessary optimality conditions
in the fractional calculus of variations \cite{book:FCV},
and by proving existence of minimizers
for fractional variational problems with right fractional operators
from the recent existence results obtained in \cite{BourdinTorres}
for left fractional variational problems (Section~\ref{sub:sec:VP}).

% ----------------------------------------

\subsection{An integration by parts formula}
\label{sub:sec:IBP}

Let $\alpha>0$ be a real number between the integers $n-1$ and $n$.
The following integration by parts formula
for the left Caputo derivative is well known
(see, \textrm{e.g.}, \cite{AGRA3}):
\begin{multline}
\label{eq:t}
\int_{a}^{b}g(x)\cdot {_a^C D_x^\alpha}f(x)dx
= \int_a^b f(x)\cdot {_x D_b^\alpha} g(x)dx\\
+ \sum_{j=0}^{n-1}\left[{_xD_b^{\alpha+j-n}}g(x)
\cdot {_xD_b^{n-1-j}} f(x)\right]_a^b,
\end{multline}
where ${_aD_x^{k}}g(x)={_aI_x^{-k}}g(x)$
and  ${_xD_b^{k}}g(x)={_xI_b^{-k}}g(x)$ if $k<0$.
Using the integration by parts formula \eqref{eq:t}
for the left Caputo derivative, we deduce, by duality,
the integration by parts formula
for the right Caputo derivative.

\begin{corollary}[Integration by parts formula]
\label{cor4.1}
Let $\alpha$ be a positive real number and
$n$ the integer satisfying $n-1 < \alpha \le n$.
The following relation holds:
\begin{multline}
\label{eq:IBP:cor}
\int_{a}^{b}g(x)\cdot {_x^C D_b^\alpha}f(x)dx
= \int_a^b f(x)\cdot {_a D_x^\alpha} g(x)dx\\
- \sum_{j=0}^{n-1}\left[{_aD_x^{\alpha+j-n}}g(x)
\cdot {_aD_x^{n-1-j}} f(x)\right]_a^b.
\end{multline}
\end{corollary}

\begin{proof}
From Definition~\ref{def:dualF},
Remark~\ref{rem:dual:dual}, and Theorem~\ref{thm:C:FD},
it follows that
\begin{equation*}
\int_{a}^{b} g(x)\cdot {_x^C D_b^\alpha}f(x)dx
= \int_{-b}^{-a}g^{\star}(s)\cdot {_{-b}^C D_s^\alpha}f^{\star}(s)ds.
\end{equation*}
By the integration by parts formula \eqref{eq:t} and Theorem~\ref{thm:RL:FD},
\begin{equation*}
\begin{split}
\int_{-b}^{-a}g^{\star}(s)\cdot {_{-b}^C D_s^\alpha}f^{\star}(s)ds
&= \int_{-b}^{-a} f^{\star}(s) \cdot {_{s}D_{-a}^\alpha}g^{\star}(s)ds\\
& \qquad + \sum_{j=0}^{n-1}\left[{_sD_{-a}^{\alpha+j-n}}g^{\star}(s)
\cdot {_sD_{-a}^{n-1-j}} f^{\star}(s)\right]_{-b}^{-a}\\
&=\int_a^b f(x)\cdot {_a D_x^\alpha} g(x)dx\\
& \qquad - \sum_{j=0}^{n-1}\left[{_aD_x^{\alpha+j-n}}g(x)
\cdot {_aD_x^{n-1-j}} f(x)\right]_a^b.
\end{split}
\end{equation*}
This concludes the proof.
\end{proof}

\begin{remark}
\label{rem4.2}
Let $0 < \alpha \le 1$. In this case $n = 1$
and \eqref{eq:IBP:cor} reduces to
\begin{equation}
\label{eq:right:alpha:0:1}
\int_{a}^{b}g(x) \cdot {_x^C D_b^\alpha}f(x)dx
= \int_a^b f(x)\cdot {_a D_x^\alpha} g(x)dx\\
- \left[{_aD_x^{\alpha-1}}g(x) \cdot f(x)\right]_a^b.
\end{equation}
If $\alpha=1$, then the integration by parts
formula \eqref{eq:right:alpha:0:1}
gives the classical relation
$$
\int_{a}^{b}g(x)\cdot (-f'(x))dx
=\int_a^b f(x)\cdot g'(x)dx-\left[g(x) \cdot f(x)\right]_a^b
$$
(recall Remark~\ref{rem:2.1}).
\end{remark}

Corollary~\ref{cor4.1} is an application of our Theorems~\ref{thm:C:FD} and \ref{thm:RL:FD}.
The integration by parts formula \eqref{eq:IBP:cor} has a crucial role
in the fractional variational calculus \cite{book:FCV}. Our proof shows that such
important result follows by duality from another
well known integration by parts formula \eqref{eq:t}, without the need
to repeat all the arguments again: formulas \eqref{eq:t} and \eqref{eq:IBP:cor}
are dual and one only needs to provide a proof to one of them,
the proof to the other following then by duality, as a corollary.

% ----------------------------------------

\subsection{Existence for right fractional variational problems}
\label{sub:sec:VP}

While the study of fractional variational problems was initiated by Riewe
in 1996/97 \cite{Riewe1,Riewe2}, including from the very beginning
problems with both left and right fractional operators, the question of existence
was only address in 2013 and only for left
fractional variational problems \cite{BourdinTorres}.
Here we show, from the duality introduced in Section~\ref{sec3},
how existence of solutions for fractional problems
of the calculus of variations involving right operators
follow from the results of \cite{BourdinTorres}.

Let $a,b$ be two real numbers such that $a<b$, let $d\in\N^*$,
where $\N^*$ denotes the set of positive integers,
and let $\left\|\cdot\right\|$ denote
the standard Euclidean norm of $\R^d$.
For any $1 \leq r \leq \infty$, denote
\begin{itemize}
\item by $\L^r := \L^r (a,b;\R^d)$ the usual space of $r$-Lebesgue
integrable functions endowed with its usual norm $\Vert \cdot \Vert_{\L^r}$;
\item by $\W^{1,r} := \W^{1,r} (a,b;\R^d)$ the usual $r$-Sobolev
space endowed with its usual norm $\Vert \cdot \Vert_{\W^{1,r}}$.
\end{itemize}
Furthermore, let $\CC := \CC ([a,b];\R^d)$ be understood as the standard space
of continuous functions and $\CC^{\infty}_c := \CC^{\infty}_c ([a,b];\R^d)$
as the standard space of infinitely differentiable functions compactly supported in $(a,b)$.
Finally, let us remind that the compact embedding
$W^{1,r}\hookrightarrow \CC$  holds for $1 < r \le +\infty$.

Let $0< \alpha <1$ and $\dot{f}$ denote the usual derivative of $f$.
Then the left and the right Caputo fractional derivatives
of order $\alpha$ are given by
\begin{equation*}
\Dl[f](t):=\IlC [\dot{f}](t)
~~\textnormal{ and }~~\Dr[f](t):=-\IrC [\dot{f}](t)
\end{equation*}
for all $t\in (a,b]$  and $t \in [a,b)$, respectively.
Let $1 < p < \infty$, $p'$ denote the adjoint of $p$,
and consider the (right) fractional variational functional
\begin{equation}
\label{eq:varFunc}
\begin{gathered}
{\LL}: {\E}\rightarrow{\R}\\
{u}\mapsto{\displaystyle
\int_a^b L\left(u(t),\Ir[u](t),\dot{u}(t),\Dr[u](t),t\right) \; dt}.
\end{gathered}
\end{equation}
Our main goal is to prove existence of minimizers for $\LL$.
An an example, consider the classical problem of linear friction:
\begin{equation}
\label{eq:lin:fric}
m\frac{d^2u}{dt^2}+\gamma\frac{du}{dt} -\frac{\partial U}{\partial u}=0, \quad \gamma>0.
\end{equation}
In 1931, Bauer proved that it is impossible to use a
variational principle to derive this linear dissipative
equation of motion with constant coefficients.
Bauer's theorem expresses the well-known belief that there is no direct
method of applying variational principles to nonconservative
systems, which are characterized by friction or other dissipative
processes. The proof of Bauer's theorem relies, however, on the tacit
assumption that all derivatives are of integer order.
With the Lagrangian
$$
L=\frac{1}{2}m \dot{u}^2-U(u)+\frac{1}{2}\gamma\left({_t^CD_b^{\frac{1}{2}}}u\right)^2
$$
one can obtain \eqref{eq:lin:fric} from the fractional variational principle.
For details see \cite{Riewe2}.

We assume that $\E$ is a weakly closed subset of
$\W^{1,p}$, $\dot{u}$ is the derivative of $u$,
and $L$ is a Lagrangian of class $\CC^1$:
\begin{equation}
\label{eq:Lagrangian}
\begin{gathered}
{L}:{(\R^d)^4 \times [a,b]}\rightarrow{\R}\\
{(x_1,x_2,x_3,x_4,t)}\mapsto {L(x_1,x_2,x_3,x_4,t).}
\end{gathered}
\end{equation}
By $\partial_i L$ we denote the partial derivatives
of $L$ with respect to its $i$th argument, $i = 1,\ldots, 5$.
We introduce the following notions of regularity and coercivity.

\begin{definition}[Regular Lagrangian]
\label{def:regular}
We say that a Lagrangian $L$ given by \eqref{eq:Lagrangian} is \emph{regular} if
\begin{itemize}
\item $L\left(u,\Ir[u],\dot{u},\Dr[u],t\right) \in \L^1$;
\item $\partial_1 L\left(u,\Ir[u],\dot{u},\Dr[u],t\right) \in \L^1$;
\item $\partial_2 L\left(u,\Ir[u],\dot{u},\Dr[u],t\right) \in \L^{p'}$;
\item $\partial_3 L\left(u,\Ir[u],\dot{u},\Dr[u],t\right) \in \L^{p'}$;
\item $\partial_4 L\left(u,\Ir[u],\dot{u},\Dr[u],t\right) \in \L^{p'}$;
\end{itemize}
for any $u\in\W^{1,p}$.
\end{definition}

\begin{definition}[Coercive functional]
\label{def:coercivity}
We say that a fractional variational functional $\LL$
is \emph{coercive} on $E$ if
\begin{equation*}
\lim\limits_{\substack{\Vert u \Vert_{\W^{1,p}}
\to \infty \\ u \in \E }} \LL(u) = +\infty.
\end{equation*}
\end{definition}

Finally, we introduce the notions of dual Lagrangian
and dual variational functional.

\begin{definition}[Dual Lagrangian]
\label{def:dualLag}
Let $L$ be a Lagrangian \eqref{eq:Lagrangian}.
The dual Lagrangian $L^{\star}$ of $L$ is defined by
\begin{gather*}
{L^{\star}}:{(\R^d)^4 \times [-b,-a]}\rightarrow{\R}\\
{(x_1,x_2,x_3,x_4,s)}\mapsto{L(x_1,x_2,-x_3,x_4,-s).}
\end{gather*}
\end{definition}

\begin{definition}[Dual variational functional]
\label{def:DualVarFunc}
We say that $\LL^{\star}$ is the dual variational functional
of ${\LL}$ given by \eqref{eq:varFunc} if
\begin{gather*}
\LL^{\star} : {\E^{\star}}\rightarrow{\R}\\
{u^{\star}}\rightarrow{\displaystyle
\int_{-b}^{-a} L^{\star}\left(u^{\star},{_aI_s^\alpha}[u^{\star}],
\dot{u^{\star}},{_a^CD_s^\alpha}[u^{\star}],s\right) \; ds},
\end{gather*}
where $E^{\star}$ is the dual space of $E$ (see \cite{Caputo})
and $L^{\star}$ is the dual Lagrangian of $L$.
\end{definition}

Note that in Definition~\ref{def:dualLag} we only put a minus sign
in the entries that correspond to the independent time variable $t$
and the classical derivative. In this way, doing the change of variable $s = -t$ in
the fractional variational functional \eqref{eq:varFunc}, we obtain the dual variational functional
of Definition~\ref{def:DualVarFunc} from the results of Section~\ref{sec3}.

\begin{proposition}
\label{prop:reg}
A Lagrangian $L$ is \emph{regular}, in the (right) sense of Definition~\ref{def:regular},
if and only if $L^{\star}$, the dual Lagrangian of $L$,
is \emph{regular} in the (left) sense of \cite[Definition~3.1]{BourdinTorres}.
\end{proposition}

\begin{proof}
The result is a direct consequence of the definition of regular Lagrangian
introduced by \cite[Definition~3.1]{BourdinTorres} for left fractional
operators and the duality results of Section~\ref{sec3}.
\end{proof}

\begin{proposition}
\label{prop:coerciv}
A fractional variational functional $\LL$ given by \eqref{eq:varFunc}
is \emph{coercive} on $E$ if and only if the corresponding
dual variational functional $\LL^{\star}$ is \emph{coercive} on $E^{\star}$.
\end{proposition}

\begin{proof}
Direct consequence of definitions.
\end{proof}

The next result asserts the boundedness of the right Riemann--Liouville
fractional integrals in the space $\L^r$.

\begin{proposition}
\label{prop:bounded}
The right Riemann--Liouville fractional integral
$\Ir$ with $\alpha >0$ is a linear and bounded operator in $\L^r$:
\begin{equation*}
\left\|\Ir[f]\right\|_{\L^r}
\leq \frac{(b-a)^{\alpha}}{\Gamma(1+\alpha)}\left\|f\right\|_{\L^r}
\end{equation*}
for all $f\in\L^r$, $1 \leq r \leq +\infty$.
\end{proposition}

\begin{proof}
The result follows by duality from the boundedness of the left
Riemann--Liouville fractional integral (see, \textrm{e.g.},
\cite[Proposition 2.1]{BourdinTorres}).
\end{proof}

Under regularity, coercivity and convexity
we can prove, from the duality of Section~\ref{sec3}
and the Tonelli theorem in \cite{BourdinTorres}
for fractional variational problems involving left operators,
a Tonelli-type theorem that ensures the existence
of a minimizer for a right fractional variational functional
$\LL$ given by \eqref{eq:varFunc}.

\begin{theorem}[Tonelli's existence theorem
for right fractional variational problems \eqref{eq:varFunc}]
\label{thmtonelli}
If
\begin{itemize}
\item $L$ is regular (Definition~\ref{def:regular});
\item $\LL$ is coercive on $E$ (Definition~\ref{def:coercivity});
\item $L(\cdot,t)$ is convex on $(\R^d)^4$ for any $t \in [a,b]$;
\end{itemize}
then there exists a minimizer for the right fractional variational problem $\LL$.
\end{theorem}

\begin{proof}
Follows from Propositions~\ref{prop:reg}, \ref{prop:coerciv} and \ref{prop:bounded}
and \cite[Theorem 3.3]{BourdinTorres}.
\end{proof}

% ----------------------------------------

\section{Conclusion}
\label{sec:conc}

In this work we developed further the theory of fractional calculus that has the definition
of two fractional derivative/intregral operators as its base: a left operator,
which is nonlocal by looking to the past/left of the current time/space,
and a right operator, which is nonlocal by looking to the future/right
of the current instant/position.
Both perspectives (left and right, causal and anti-causal)
make all sense in many applications, like Signal Processing,
where bilateral operators, like the bilateral Laplace transform,
and right and left functions have a central role \cite{M:O:P:T}.
We trust that the fractional signal processing community will gain new interested
people and application areas with our results because we develop a new set of simple tools
that allows to substitute the usual procedure of repeating the arguments
for both left and right cases by simple duality,
deducing directly one of the situations from the other.
Since the involved mathematical analysis turns out to be simpler and less involved
with our duality technique, it is natural to believe in the success of our new fractional methodology
in applications, where it can allow simpler and thus better models.
More precisely, the new way we propose to look to a right fractional operator,
as the dual of the corresponding left operator, may bring new ways
of dealing with practical systems. For future work,
it would be interesting to generalize our results
to other classes of fractional operators like Gr\"unwald--Letnikov
derivatives \cite{MR2806728} or discrete fractional operators
like those of \cite{MyID:179}. Other direction of research
concerns the implications of duality in the expressions
for the transfer functions and their regions of convergence.
This is under investigation and will be addressed elsewhere.

% ----------------------------------------

\section*{Acknowledgements}

This work was partially supported by Portuguese funds through the
\emph{Center for Research and Development in Mathematics and Applications} (CIDMA),
and \emph{The Portuguese Foundation for Science and Technology} (FCT),
within project PEst-OE/MAT/UI4106/2014.
Torres was also supported by the FCT project
PTDC/EEI-AUT/1450/2012, co-financed by FEDER under POFC-QREN
with COMPETE reference FCOMP-01-0124-FEDER-028894.
The authors are very grateful to four anonymous referees
for valuable remarks and comments, which
significantly contributed to the quality of the paper.

% ----------------------------------------

% ----------------------------------------

\end{document}